\documentclass[11pt]{article}

\usepackage{ amsthm, amsmath,amssymb}
\usepackage{epstopdf}

\def \[{\begin{equation}}
\def \]{\end{equation}}

\def\bdes{\begin{description}}
\def\edes{\end{description}}
\def\benu{\begin{enumerate}}
\def\eenu{\end{enumerate}}
\def\bitm{\begin{itemize}}
\def\eitm{\end{itemize}}

\def\R{{\sl I\kern-3.2pt R}}

\def\sqr#1#2{{\vcenter{\hrule height .#2pt
      \hbox{\vrule width .#2pt height#1pt \kern#1pt\vrule width.#2pt}
                       \hrule height.#2pt}}}

\def\x{{\cal X}}

\newtheorem{proposition}{Proposition}[section]

\newtheorem{remark}{Remark}[section]

\newtheorem{thm}{Theorem}[section]
\newtheorem{pro}[thm]{Proposition}
\newtheorem{lem}[thm]{Lemma}

\begin{document}
\title{{\Large \bf   A Pohozaev Identity for the Fractional H$\acute{e}$non System }}
\author{Pei Ma\;\;,
Fengquan Li\;\;,
Yan Li\thanks{Corresponding author.}}

\date{\today}
\maketitle

\begin{abstract}
In this paper, we study the Pohozaev identity associated with a H$\acute{e}$non-Lane-Emden system involving the fractional Laplacian:
\begin{eqnarray*}\left\{
\begin{aligned}
 &(-\triangle)^su=|x|^av^p,~~~~&x\in\Omega,\\
 &(-\triangle)^sv=|x|^bu^q,~~~~&x\in\Omega,\\
 &u=v=0,~~~~&x\in R^n\backslash\Omega,
 \end{aligned}
 \right.
\end{eqnarray*}
 in a star-shaped and bounded domain $\Omega$ for $s\in(0,1)$. As an application of our identity, we deduce the nonexistence of positive solutions in the critical and supercritical cases.
\end{abstract}

\section{Introduction}
We study the following system
\begin{eqnarray}\label{1.1}
\left\{
\begin{aligned}
 &(-\triangle)^{s}u=|x|^av^p,~~~~~~&x\in\Omega,\\
 &(-\triangle)^{s}v=|x|^bu^q,~~~~~~&x\in\Omega,\\
 &u=0, v= 0,~~~~~~&x\in R^n  \backslash \Omega,
 \end{aligned}
 \right.
\end{eqnarray}
in a star-sharped and bounded domain $\Omega\subset R^n$ with $C^{1,1}$ boundary. Assume that $s\in(0,1)$, $pq>1$, $p$, $q$, $a$, $b$ $\geq0$, $n\geq1$. We will prove the nonexistence of positive solutions for
system (\ref{1.1}) in the critical and supercritical cases $\frac{n+a}{p+1}+\frac{n+b}{q+1}\leq n-2s$.

The fractional Laplacian in $R^n$ is a nonlocal pseudo-differential operator, assuming the form
\begin{eqnarray}\nonumber
(-\triangle)^{s}u&=&C_{n,s}PV \int_{R^n}\frac{u(x)-u(y)}{|x-y|^{n+2s}}dy\\
&=&C_{n,s}\lim_{\varepsilon\rightarrow 0} \int_{R^n\setminus B_\varepsilon(0)}\frac{u(x)-u(y)}{|x-y|^{n+2s}}dy,\label{1.2}
\end{eqnarray}
where $C_{n,s}$ is a normalization constant.

Let
\begin{eqnarray*}
\mathcal{L}_s=
\{u:R^n\rightarrow R\mid\int_{R^n}\frac{|u(x)|}{1+|x|^{n+2s}}dx<\infty\}.
\end{eqnarray*}
Obviously, the integral in (\ref{1.2}) is well defined for $u \in \mathcal{L}_\alpha\cap C^{1,1}_{loc}$.

In this paper, we consider solutions in the weak sense. Define
\begin{equation*}
  H^s(R^n)=\{u:R^n\rightarrow R\mid\int_{R^n}|\xi|^{2s}|\hat{u}(\xi)|^2d\xi<+\infty\}.
\end{equation*}

Let $H_0^s$ be the completion of $C_0^\infty(\Omega)$ under this norm
$$\|u\|_{H_0^s}^2=\int_{R^n}|\xi|^{2s}|\hat{u}(\xi)|^2d\xi.$$

We say that $u\in H_0^s(\Omega)$ is a weak solution of
\begin{eqnarray*}
\left\{
\begin{aligned}
 &(-\triangle)^{s}u=f(x,u),~~~~~~&x\in\Omega,\\
 &u=0, ~~~~~~&x\in R^n  \backslash \Omega,
 \end{aligned}
 \right.
\end{eqnarray*}
if for all $\varphi\in C_0^\infty(\Omega)$,
$$\int_{R^n}(-\triangle)^\frac{s}{2}u(-\triangle)^\frac{s}{2}\varphi dx=\int_\Omega f(x,u)\varphi dx,$$
where
$$\int_{R^n}(-\triangle)^\frac{s}{2}u(-\triangle)^\frac{s}{2}\varphi dx=\int_{R^n}|\xi|^{2s}\hat{u}(\xi)\hat{\varphi}(\xi) d\xi.$$

The single equation below has been well studied by many mathematicians:
\begin{eqnarray}\label{4.1}
\left\{
\begin{aligned}
 &(-\triangle)^{s}u=|x|^\alpha u^p,~~~~~~&x\in\Omega,\\
 &u=0,~~~~~~&x\in R^n  \backslash \Omega.
 \end{aligned}
 \right.
 \end{eqnarray}
It is called the Hardy-type for $\alpha<0$, due to its relation to the Hardy-Sobolev inequality. For $\alpha>0$, the equation is known as the H$\acute{e}$non-type, because it was introduced by H$\acute{e}$non in 1973 (\cite{HM}) for the study of stellar systems.

In \cite{RS}, the authors considered
 \begin{eqnarray}\label{eee1}
\left\{
\begin{aligned}
&(-\triangle)^su=f(u), & x \in \Omega,\\
 &u\equiv 0,& x \in R^n\setminus\Omega.
 \end{aligned}
 \right.
 \end{eqnarray}
 There they proved
  \begin{proposition}(\textbf{Pohozaev identity for single equation})
Let $\Omega$ be a bounded and $C^{1,1}$ domain and $u$ is a solution of  (\ref{eee1}).
Then the following identity holds
$$\int_{\Omega}(x\cdot \nabla u)(-\triangle)^s udx=\frac{2s-n}{2}\int_{\Omega}u(-\triangle)^s u dx-\frac{\Gamma(1+s)^2}{2}\int_{\partial\Omega}(\frac{u}{\delta^s})^2(x\cdot\nu)d\sigma,$$
where $\nu$ is the unit outward normal to $\partial \Omega$ at $x$, and $\Gamma$ is the Gamma function.
 \end{proposition}

Using the Pohozaev identity above, they were able to deduce the nonexistence of nontrivial solutions.

In this paper, when investigating the more general system (\ref{1.1}), we find a new phenomenon that there are two more terms on the right hand side of the identity:

\begin{thm}\label{thm2}(\textbf{Pohozaev identity for systems})
Let $\Omega$ be a bounded and $C^{1,1}$ domain, $\delta(x)=dist(x,\partial\Omega)$. Assume that  $(u,v)$ is a pair of solution in (\ref{1.1}), then the following identity holds
\begin{eqnarray}\nonumber
&&\int_\Omega(x\cdot\nabla v)(-\triangle)^s udx\\ \nonumber
&=&\frac{2s-n}{2}\int_\Omega v(-\triangle)^sudx-\frac{\Gamma(1+s)^2}{2}\int_{\partial\Omega}\frac{u}{\delta^s}\frac{v}{\delta^s}(x\cdot\nu)d\sigma\\ \nonumber
&-&\frac{1}{2}\int_{\Omega}\bigg(x\nabla U(x)V(x)-xU(x)\nabla V(x)\bigg)dx
-\frac{1}{2}\int_{R^n\backslash\Omega}\bigg(x\nabla U(x)V(x)-xU(x)\nabla V(x)\bigg)dx,\\
\label{1.10}
\end{eqnarray}
where $U(x)=(-\triangle)^{s/2}u(x)$ and $V(x)=(-\triangle)^{s/2}v(x)$.
\end{thm}

It's worth pointing out that new difficulty
arises in the application of this new Pohozaev identity for the system (\ref{1.1}).
Because of the uncertainty of the sign for the last two terms in (\ref{1.10}), this identity does not come as handy as its counterpart in  \cite{RS}. To overcome this, briefly speaking, we partner (\ref{1.10})
with a second identity to get rid of the undeterminable terms.
We will give more detailed account in later proofs.

Our goal here is to establish similar nonexistence result for
system (\ref{1.1}). As an application of Theorem (\ref{thm2}), we have

\begin{thm}\label{thm1}
Assume that $\Omega$ is a bounded and star-shaped domain with $C^{1,1}$ boundary. Then system (\ref{1.1}) admits no positive bounded solution in both critical and supercritical cases
$\frac{n+a}{p+1}+\frac{n+b}{q+1}\leq n-2s$.
\end{thm}

\begin{remark}  Notice that in (\ref{1.10}),
\begin{eqnarray*}
&&\int_{\Omega}x\nabla U(x)V(x)-xU(x)\nabla V(x)dx
+\int_{R^n\backslash\Omega}\bigg(x\nabla U(x)V(x)-xU(x)\nabla V(x)\bigg)dx\\
&\neq&\int_{R^n}\bigg(x\nabla U(x)V(x)-xU(x)\nabla V(x)\bigg)dx,
\end{eqnarray*}
because the integrand is highly singular on $\partial \Omega$ (some kind of Delta measure). To illustrate this point, consider the simple example,
 \begin{eqnarray*}
 f(t)=\left\{
\begin{aligned}
 t+1,~~~t\geq0,\\
 t,~~~t<0,
 \end{aligned}
 \right.
 \end{eqnarray*}
 then $f'(t)=1+\delta(t)$, where $\delta$ is the Delta function. Since $\int_{-1}^0f'(t)dt=\lim\limits_{\varepsilon\rightarrow0}\int_{-1}^{-\varepsilon} f'(t)dt=1$ and $\int^{1}_0f'(t)dt=\lim\limits_{\varepsilon\rightarrow0}\int^{1}_\varepsilon f'(t)dt=1$, then
\begin{equation*}
\int_{-1}^0f'(t)dt+\int_0^1f'(t)dt=2\neq\int_{-1}^1f'(t)dt=3.
\end{equation*}

Let $\Omega_\varepsilon=\{x\in R^n\mid\text{dist}(x,\Omega)<\varepsilon\}$, then the integral under consideration is in the following sense
\begin{equation*}
\int_{R^n\backslash\Omega}\bigg(x\nabla U(x)V(x)-xU(x)\nabla V(x)\bigg)dx=\lim_{\varepsilon\rightarrow0}\int_{R^n\backslash\Omega_\varepsilon}\bigg(x\nabla U(x)V(x)-xU(x)\nabla V(x)\bigg)dx,
\end{equation*}
and
$$\int_\Omega\big(x\nabla U(x)V(x)-xU(x)\nabla V(x)\big)dx$$
 is defined similarly.
\end{remark}

\begin{remark}
In the fractional Pohozaev identity, the functions $u/{\delta^s}|_{\partial\Omega}$ and $v/{\delta^s}|_{\partial\Omega}$ play the roles of $\partial u/\partial \nu$ and $\partial v/\partial \nu$ in the classical Pohozaev identity. Surprisingly, from a nonlocal problem we obtain an identity with a boundary term (an integral over $\partial\Omega$) which is completely local.
\end{remark}

In recent years, the fractional H$\acute{e}$non system has
 received a lot of attention.

In \cite{DZ}, under some regularity conditions, the authors showed that for $\alpha \in (0,2)$ and $\beta, \gamma>0$, the differential system
 \begin{eqnarray*}
\left\{
\begin{aligned}
 &(-\triangle)^{\frac{\alpha}{2}}u=|x|^\beta v^p,~~~~~~&x\in R^n,\\
 &(-\triangle)^{\frac{\alpha}{2}}v=|x|^\gamma u^q,~~~~~~&x\in R^n,\\
 &u,v\geq 0,
 \end{aligned}
 \right.
 \end{eqnarray*}
 is equivalent to an integral system
 \begin{eqnarray*}
\left\{
\begin{aligned}
 &u(x)=C_1\int_{R^n}\frac{|y|^\beta v^p(y)}{|x-y|^{n-\alpha}}dy\\
 &v(x)=C_2\int_{R^n}\frac{|y|^\gamma u^q(y)}{|x-y|^{n-\alpha}}dy.
 \end{aligned}
 \right.
 \end{eqnarray*}
 Then using the method of moving planes in integral forms,
 they proved some Liouville type theorems.

 In \cite{LNZ2}, the authors considered a weighted system
 \begin{eqnarray*}
\left\{
\begin{aligned}
 &(-\triangle)^{\frac{\alpha}{2}}u=|x|^{-s} u^p,~~~~~~&x\in R^n,\\
 &(-\triangle)^{\frac{\alpha}{2}}u=|x|^{-t} u^q,~~~~~~&x\in R^n,
 \end{aligned}
 \right.
 \end{eqnarray*}
 with $\alpha \in (0, n)$, $0\leq s, t< \alpha$.
 They first established the equivalence between the differential system and an integral system:
 \begin{eqnarray*}
\left\{
\begin{aligned}
 &u(x)=\int_{R^n}\frac{ v^p(y)}{|x-y|^{n-\alpha}|y|^s}dy\\
 &v(x)=\int_{R^n}\frac{ u^q(y)}{|x-y|^{n-\alpha}|y|^t}dy.
 \end{aligned}
 \right.
 \end{eqnarray*}
 Then they proved radial symmetry
 in the critical case $\frac{n-s}{q+1}+\frac{n-t}{p+1}=n-\alpha$,  and the nonexistence in the subcritical case  for positive solutions. For more such details on the application of the method of moving
 planes in integral forms. please see \cite{CL2}, \cite{LY},  \cite{SZ1} and the reference therein.

The paper is organized as follows.  We first give some essential ingredients for the proofs of the main theorems. Then in Section 3, we derive Theorem \ref{thm2} and \ref{thm1}.

\section{Preliminary}

\begin{pro}\label{pro1}\cite{RS}
Let $\Omega$ be a bounded and $C^{1,1}$ domain and $u$ be a function such that $u=0$ in $R^n\backslash\Omega$ and that $u$ satisfies
\begin{enumerate}
\item The function $u/{\delta^s}|_\Omega$ can be continuously extended to $\bar{\Omega}$. Moreover, there exists $\alpha\in(0,1)$ such that $u/{\delta^s}\in C^\alpha(\bar{\Omega})$. In addition, for all $\beta\in [\alpha,s+\alpha]$, it holds the estimate
\begin{equation*}
[u/\delta^s]_{C^\beta(\{x\in\Omega:\delta\geq\rho\})}\leq C\rho^{\alpha-\beta}\\ for\ all\ \rho\in(0,1).
\end{equation*}
\end{enumerate}
Then there exists a $C^\alpha(R^n)$ extension $\bar{u}$ of $u/{\delta^s}|_\Omega$ such that
\begin{equation}\label{1.3}
(-\triangle)^{s/2}u(x)=c_1\{\log^-\delta(x)+c_2\chi_\Omega(x)\}\bar{u}(x)+h(x) \ in\ R^n,
\end{equation}
where $h$ is a $C^\alpha(R^n)$ function, $\log^-t=\min\{\log t,0\}$,
\begin{eqnarray*}
c_1=\frac{\Gamma(1+s)\sin(\frac{\pi s}{2})}{\pi},~~~~~c_2=\frac{\pi}{\tan(\frac{\pi s}{2})}.
\end{eqnarray*}
\end{pro}

\begin{pro}\label{pro2}
Let $A$, $B$, $C$ and $D$ be real numbers, and
\begin{eqnarray*}
\phi(t)=A\log^-|t-1|+B\chi_{[0,1]}(t)+h_0(t),\\
\psi(t)=C\log^-|t-1|+D\chi_{[0,1]}(t)+h_1(t),
\end{eqnarray*}
where $\log^-t=\min\{\log t,0\}$, $h_0$ and $h_1$ are functions satisfying, for some constants $\alpha$ and $\gamma$ in $(0,1)$, and $C_0>0$, the following conditions:

   (i) $\|h\|_{C^\alpha([0,\infty))}\leq C_0$.

   (ii) For all $\beta\in[\gamma,1+\gamma]$
\begin{equation*}
\|h\|_{C^\beta((0,1-\rho)\cup(1+\rho,2))}\leq C_0\rho^{-\beta} \,\, for \,\,all\,\, \rho\in(0,1).
\end{equation*}

   (iii) $|h^{'}(t)|\leq C_0t{-2-\gamma}$ and $|h^{''}(t)|\leq C_0t{-3-\gamma}$ for all $t>2$.
Then,
\begin{eqnarray*}
-\frac{d}{d\lambda}\bigg|_{\lambda=1^+}\int_0^\infty\phi(\lambda t)\psi(\frac{t}{\lambda})dt&=&-AC\pi^2-BD
-\int_0^1[th_0'(t)h_1(t)-th_0(t)h_1'(t)]dt\\
&&-\int_1^\infty[th_0'(t)h_1(t)-th_0(t)h_1'(t)]dt.
\end{eqnarray*}
\end{pro}

To prove the Proposition \ref{pro2}, we need the following lemmas.

\begin{lem}\label{lem1}
Let $h_0$, $h_1$ be functions satisfying (i), (ii) and (iii) in Proposition $\ref{pro2}$, $\lambda\in(1,3/2)$ and $\varepsilon\in(0,1)$ such that $\varepsilon/2>\lambda-1$. Let $\alpha$, $\gamma$ and $C_0$ be the constants appearing in (i)-(ii)-(iii). Then,
\begin{eqnarray*} |h_0(\lambda t)h_1(\frac{t}{\lambda})-h_0(t)h_1(t)|\leq\left\{
\begin{aligned}
&C|\lambda-1|^\alpha,&&t\in(1-\varepsilon,1+\varepsilon),\\
&C\rho^{-1-\gamma}|\lambda-1|^{\gamma+1}+|\phi'(1)(\lambda-1)|,&&t\in(0,1-\varepsilon)\cup(1+\varepsilon,2),\\
&C|\lambda-1|^2t^{-1-\gamma}+|\phi'(1)(\lambda-1)|,&&t\in (2,\infty),
 \end{aligned}
 \right.
\end{eqnarray*}
where the constant $C$ depends only on $C_0$ and $\phi'(1)=th_0'(t)h_1(t)-th_0(t)h_1'(t)$.
\end{lem}

\begin{proof} For $t\in(1-\varepsilon,1+\varepsilon)$, from the fact that $\|h_0\|_{C^\alpha([0,\infty))}$, $\|h_1\|_{C^\alpha([0,\infty))}\leq C_0$ and $\|h_0\|_{L^\infty}(R^n)$, $\|h_1\|_{L^\infty}(R^n)\leq C_0$, we obtain
\begin{eqnarray*}
&&|h_0(\lambda t)h_1(\frac{t}{\lambda})-h_0(t)h_1(t)|\\
&=&|h_0(\lambda t)h_1(\frac{t}{\lambda})-h_0(t)h_1(\frac{t}{\lambda})+h_0(t)h_1(\frac{t}{\lambda})-h_0(t)h_1(t)|\\
&=&|h_0(\lambda t)-h_0(t)||h_1(\frac{t}{\lambda})|+|h_0(t)||(h_1(\frac{t}{\lambda})-h_1(t))|\\
&\leq&C|\lambda t-t|^\alpha+C|\frac{t}{\lambda}-t|^\alpha\\
&\leq&C|\lambda-1|^\alpha,
\end{eqnarray*}

For $t\in (0,1-\varepsilon)\cap(1+\varepsilon,2)$ and $\mu\in[1,\lambda]$, defined
\begin{equation*}
\phi(\mu)=h_0(\mu t)h_1(\frac{t}{\mu})-h_0(t)h_1(t)
\end{equation*}
By the mean value theorem,
 $$\phi(\lambda)=\phi(1)+\phi^{'}(\mu)(\lambda-1)
 \,\,for\,\, some\,\, \mu\in(1,\lambda).$$
Obviously $\phi(1)=0$. Note that
\begin{equation*}
\phi^{'}(\mu)=th_0^{'}(\mu t)h_1(\frac{t}{\mu})-\frac{t}{\mu^2}h_0(\mu t)h_1^{'}(\frac{t}{\mu}),
\end{equation*}
then
\begin{eqnarray}\label{1.8}\nonumber
|\phi^{'}(\mu)|&=&|\phi^{'}(\mu)-\phi^{'}(1)+\phi^{'}(1)|\\
&\leq&|\phi^{'}(\mu)-\phi^{'}(1)|+|\phi^{'}(1)|.
\end{eqnarray}
Thus, using the bounds from (ii) with $\beta$ replaced by $\gamma$, $1$ and $1+\gamma$,
\begin{eqnarray}\label{1.9}\nonumber
&&|\phi^{'}(\mu)-\phi^{'}(1)|\\ \nonumber
&=&|th_0^{'}(\mu t)h_1(\frac{t}{\mu})-\frac{t}{\mu^2}h_0(\mu t)h_1^{'}(\frac{t}{\mu})-th_0^{'}(t)h_1(t)-th_0(t)h_1^{'}(t)|\\ \nonumber
&=&t|h_0^{'}(\mu t)-h_0^{'}(t)||h_1^{'}(\frac{t}{\mu})|+t|h_0^{'}(t)||h_1^(\frac{t}{\mu})-h_1(t)|+t|h_1^{'}(t)-h_1{'}(\frac{t}{\mu})|
\frac{|h_0(\mu t)|}{\mu^2}\\nonumber
&&+th_1^{'}(t)|h_0(t)-\frac{h_0(\mu t)}{\mu^2}|\\ \nonumber
&\leq&tC\rho^{-1-\gamma}|\mu t-t|^\gamma+Ct\rho^{-1-\gamma}|\frac{t}{\mu}-t|^\gamma+Ct|t-\frac{t}{\mu}|^\gamma \rho^{-1-\gamma}+Ct\rho^{-1}[|t-\mu t|^\gamma\rho^{-\gamma}+C]\\
&=&C\rho^{-1-\rho}|\mu-1|^\gamma,
\end{eqnarray}
where $\rho=\min\{|\mu t-1|,|t-1|,|\frac{t}{\mu}-1|\}$, then
\begin{equation*}
|\phi^{'}(\mu)|\leq C\rho^{-1-\rho}|\mu-1|^\gamma+|\phi^{'}(1)|.
\end{equation*}

Finally, for $t\in(2,\infty)$, with a similar argument as in (\ref{1.9}) and using the bound (iii) instead of (ii), we obtain
\begin{equation*}
|\phi^{'}(\mu)|\leq Ct^{-1-\rho}|\lambda-1|^2+|\phi^{'}(1)|.
\end{equation*}
This completes the proof.
\end{proof}

\begin{lem}\label{lem2.4}
Let $h_0$ and $h_1$ be the function satisfying (i), (ii) and (iii), then
\begin{eqnarray*}
&&\frac{d}{d\lambda}\bigg|_{\lambda\rightarrow 1^+}\int_0^\infty h_0(\lambda t)h_1(\frac{t}{\lambda})dt\\
&=&\int_0^{1}th_0'(t)h_1(t)-th_0(t)h_1'(t)dt
+\int^\infty_{1}th_0'(t)h_1(t)-th_0(t)h_1'(t)dt.
\end{eqnarray*}
\end{lem}

\begin{proof}
Direct computations yield
\begin{eqnarray*}
&&\frac{d}{d\lambda}\bigg|_{\lambda\rightarrow1^+}\int_0^{1-\varepsilon}h_0(\lambda t)h_1(\frac{t}{\lambda})dt\\
&=&\int_0^{1-\varepsilon}\frac{d}{d\lambda}\bigg|_{\lambda\rightarrow1^+}h_0(\lambda t)h_1(\frac{t}{\lambda})dt\\
&=&\int_0^{1-\varepsilon}\lim_{\lambda\rightarrow1^+}\frac{h_0(\lambda t)h_1(\frac{t}{\lambda})-h_0(t)h_1(t)}{\lambda-1}dt\\
&=&\int_0^{1-\varepsilon}\lim_{\lambda\rightarrow1^+}\phi^{'}(\mu)dt\\
&=&\int_0^{1-\varepsilon}\lim_{\lambda\rightarrow1^+}[\phi^{'}(\mu)-\phi^{'}(1)+\phi^{'}(1)]dt\\
&=&\int_0^{1-\varepsilon}\lim_{\lambda\rightarrow1^+}[\phi^{'}(\mu)-\phi^{'}(1)]dt+
\int_0^{1-\varepsilon}\phi'(1)dt.\\
\end{eqnarray*}
Recall for $\mu\in(1,\lambda)$, $\lambda-1<\frac{\varepsilon}{2}$. Let
$$\rho=\min\{|\mu t-1|,|t-1|,|\frac{t}{\mu}-1|\}$$
For $t\in(0,1-\varepsilon)$, if $t\mu<1$, then $\rho=|1-t\mu|$; if $t\mu>1$, it is easy to see $\rho=|t\mu-1|$.
Therefore,
\begin{eqnarray*}
&&\int_0^{1-\varepsilon}\lim_{\lambda\rightarrow1^+}[\phi'(\mu)-\phi'(1)]dt\\
&\leq&\int_0^{1-\varepsilon}\lim_{\lambda\rightarrow1^+}\rho^{-1-\gamma}|\mu-1|^\gamma dt\\
&\leq&\lim_{\lambda\rightarrow1^+}{(\lambda-1)}^\gamma\int_0^{1-\varepsilon}{(1-\mu t)}^{-1-\gamma}dt\\
&=& \lim_{\lambda \rightarrow1^+}(\lambda-1)^\gamma\frac{(1-\mu t)^{-\gamma}}{\gamma \mu}\bigg|_0^{1-\varepsilon}\\
&=&\lim_{\lambda\rightarrow1^+}(\lambda-1)^\gamma[(1-(1-\varepsilon)\mu)^{-\gamma}-1]\frac{1}{\gamma\mu}\\
&=&\lim_{\lambda\rightarrow1^+}\bigg[(\frac{\lambda-1}{1-(1-\varepsilon)\mu})^\gamma-\frac{(\lambda-1)^\gamma}{\gamma\mu}
\bigg]\\
&=&\lim_{\lambda\rightarrow1^+}(\frac{\lambda-1}{1-(1-\varepsilon)\mu})^\gamma\\
&\leq&\lim_{\lambda\rightarrow1^+}(\frac{1}{\mu(\lambda-1)^{\theta-1}-1})^\gamma\\
&=& 0,
\end{eqnarray*}
here we take $\varepsilon=(\lambda-1)^\theta$ and $\theta\in(0,1)$.
Then we have
\begin{equation*}
\frac{d}{d\lambda}\bigg|_{\lambda\rightarrow1^+}\int_0^{1-\varepsilon}h_0(\lambda t)h_1(\frac{t}{\lambda})dt=\int_0^{1-\varepsilon}\phi'(1)dt.
\end{equation*}
We can obtain the same conclusion for $t\in(1+\varepsilon,\infty)$ by similar argument, that is,
\begin{equation*}
\frac{d}{d\lambda}\bigg|_{\lambda\rightarrow1^+}\int^\infty_{1+\varepsilon}h_0(\lambda t)h_1(\frac{t}{\lambda})dt=\int^\infty_{1+\varepsilon}\phi'(1)dt.
\end{equation*}
For $t\in (1-\varepsilon,1+\varepsilon)$,
\begin{eqnarray*}
&&\frac{d}{d\lambda}\bigg|_{\lambda\rightarrow1^+}\int^{1+\varepsilon}_{1-\varepsilon}h_0(\lambda t)h_1(\frac{t}{\lambda})dt\\
&=&\lim_{\lambda\rightarrow1^+}\int^{1+\varepsilon}_{1-\varepsilon}\frac{h_0(\lambda t)h_1(\frac{t}{\lambda})-h_0(t)h_1(t)}{\lambda-1}dt\\
&\leq&C\lim_{\lambda\rightarrow1^+}\int^{1+\varepsilon}_{1-\varepsilon}|\lambda-1|^{\alpha-1}dt\\
&=&C\lim_{\lambda\rightarrow1^+}|\lambda-1|^{\alpha-1}\varepsilon,\\
\end{eqnarray*}
where $\varepsilon=(\lambda-1)^\theta$ and $\theta\in(0,1)$ such that $\alpha+\theta>1.$
Finally, we prove that
\begin{eqnarray*}
&&\frac{d}{d\lambda}\bigg|_{\lambda\rightarrow1^+}\int^{\infty}_0h_0(\lambda t)h_1(\frac{t}{\lambda})dt\\
&\leq&\int^\infty_{1}\phi'(1)dt+
\int_0^{1}\phi'(1)dt.
\end{eqnarray*}
\end{proof}

\begin{lem}\label{lem2.5}Let $\log^-t=\min\{\log t,0\}$, then
\begin{eqnarray}
&&\frac{d}{d\lambda}\bigg|_{\lambda=1} \int_0^\infty\log^-|\lambda t-1|\log^-|\frac{t}{\lambda}-1|dt=-\pi^2.
\end{eqnarray}
\end{lem}

\begin{proof}Let $\varepsilon=(\lambda-1)^\theta$ as given in Lemma \ref{lem2.4}, for $t\in(0,1-\varepsilon)$,
\begin{eqnarray*}
&&\frac{d}{d\lambda}\bigg|_{\lambda=1} \int_0^{1-\varepsilon}\log^-|\lambda t-1|\log^-|\frac{t}{\lambda}-1|dt\\
&=&\int_0^{1-\varepsilon}\frac{d}{d\lambda}\bigg|_{\lambda=1}\{\log^-|\lambda t-1|\log^-|\frac{t}{\lambda}-1|\}dt\\
&=&\int_0^{1-\varepsilon}\bigg[\frac{t}{|t-1|}\log^-|t-1|-\frac{t}{|t-1|}\log^-|t-1|]dt\\
&=&0.
\end{eqnarray*}
For $t\in (1+\varepsilon,\frac{2}{\lambda})$, through the same argument as before, we derive that
\begin{equation*}
\frac{d}{d\lambda}\bigg|_{\lambda=1} \int^{\frac{2}{\lambda}}_{1+\varepsilon}\log^-|\lambda t-1|\log^-|\frac{t}{\lambda}-1|dt=0.
\end{equation*}
From Lemma 4.1 in $\cite{RS}$, we have
\begin{eqnarray*}
 &&\frac{d}{d\lambda}\bigg|_{\lambda=1} \int_0^\infty\log|\lambda t-1|\log|\frac{t}{\lambda-1}|dt\\
 &=&\frac{d}{d\lambda}\bigg|_{\lambda=1} \int_0^\frac{2}{\lambda}\log|\lambda t-1|\log|\frac{t}{\lambda-1}|dt\\
 &=&-\pi^2.
\end{eqnarray*}
Then
\begin{eqnarray*}
  &&\frac{d}{d\lambda}\bigg|_{\lambda=1} \int_{1-\varepsilon}^{1+\varepsilon}\log|\lambda t-1|\log|\frac{t}{\lambda-1}|dt\\
  &=&\frac{d}{d\lambda}\bigg|_{\lambda=1}\bigg\{\int_0^\infty\log|\lambda t-1|\log|\frac{t}{\lambda-1}|dt\\
  &&-\int_0^{1-\varepsilon}\log^-|\lambda t-1|\log^-|\frac{t}{\lambda}-1|dt- \int^{\frac{2}{\lambda}}_{1+\varepsilon}\log^-|\lambda t-1|\log^-|\frac{t}{\lambda}-1|\bigg\}dt\\
  &=&-\pi^2.
\end{eqnarray*}
\end{proof}

\begin{lem}\label{lem2.6}Let $\log^-t=\min\{\log t,0\}$ and $\chi$ be the characteristic function, then
\begin{eqnarray}\nonumber
&&\frac{d}{d\lambda}\bigg|_{\lambda=1} \int_0^\infty\log^-|\lambda t-1|\chi_{[0,1]}(\frac{t}{\lambda})+\log^-|\frac{t}{\lambda}-1|\chi_{[0,1]}(\lambda t)dt=0,
\end{eqnarray}
and
\begin{eqnarray}\nonumber
&&\frac{d}{d\lambda}\bigg|_{\lambda=1} \int_0^\infty \chi_{[0,1]}(\lambda t)\chi_{[0,1]}(\frac{t}{\lambda})dt=-1.
\end{eqnarray}
\end{lem}

\begin{proof} Let $\varepsilon=(\lambda-1)^\theta$  as given in Lemma \ref{lem2.4}, for $t\in(1-\varepsilon,1+\varepsilon),$
\begin{eqnarray}\nonumber
&&\frac{d}{d\lambda}\bigg|_{\lambda=1} \int_{1-\varepsilon}^{1+\varepsilon}\log^-|\lambda t-1|\chi_{[0,1]}(\frac{t}{\lambda})dt\\ \nonumber
&=&\frac{d}{d\lambda}\bigg|_{\lambda=1}\int_{1-\varepsilon}^{\lambda}\log|1-\lambda t|dt\\ \nonumber
&=&\frac{d}{d\lambda}\bigg|_{\lambda=1}\frac{1}{\lambda}\int_{1-\lambda^2}^{1-\lambda(1-\varepsilon)}\log|y|dy,\quad y=1-\lambda t\\ \nonumber
&=&\frac{d}{d\lambda}\bigg|_{\lambda=1}\frac{1}{\lambda}\bigg[\log|y|y\bigg|_{1-\lambda^2}^{1-\lambda(1-\varepsilon)}-\frac{1}{|y|}
(\lambda^2-\lambda(1-\varepsilon))\bigg]\\
&=&\varepsilon+2\log 0-\log|\varepsilon|,
\end{eqnarray}
and
\begin{eqnarray}\nonumber
&&\frac{d}{d\lambda}\bigg|_{\lambda=1}\int_{1-\varepsilon}^{1+\varepsilon}\log^-|\frac{t}{\lambda}-1|\chi_{[0,1]}(\lambda t)dt\\ \nonumber
&=&\frac{d}{d\lambda}\bigg|_{\lambda=1}\int_{1-\varepsilon}^{\frac{1}{\lambda}}\log|1-\frac{t}{\lambda}|dt\\ \nonumber
&=&\frac{d}{d\lambda}\bigg|_{\lambda=1}\lambda\int_{1-\frac{1}{\lambda^2}}^{1-\frac{1-\varepsilon}{\lambda}}\log|y|dy\\ \nonumber
&=&\frac{d}{d\lambda}\bigg|_{\lambda=1}\lambda\bigg[\log|y|y\bigg|_{1-\frac{1}{\lambda^2}}^{1-\frac{1-\varepsilon}{\lambda}}-
\frac{1}{|y|}y(\frac{1}{\lambda^2}-\frac{1-\varepsilon}{\lambda})\bigg]\\
&=&\varepsilon-2\log 0+\log |\varepsilon|.
\end{eqnarray}

For $t\in[0,1-\varepsilon),$ we can exchange $\frac{d}{d\lambda}\big|_{\lambda=1}$ with the integral sign and show that
\begin{eqnarray*}
&&\frac{d}{d\lambda}\bigg|_{\lambda=1} \int^{1-\varepsilon}_0\log^-|\lambda t-1|\chi_{[0,1]}(\frac{t}{\lambda})+\log^-|\frac{t}{\lambda}-1|\chi_{[0,1]}(\lambda t)dt\\ \nonumber
&=&\int^{1-\varepsilon}_0\frac{d}{d\lambda}\bigg|_{\lambda=1}\log|1-\lambda t|+\log|1-\frac{t}{\lambda}|dt\\ \nonumber
&=&\int^{1-\varepsilon}_0\frac{d}{d\lambda}\bigg|_{\lambda=1}\log|1-\lambda t|dt\\ \nonumber
&=&\int^{1-\varepsilon}_0(\frac{-t}{1-t}+\frac{t}{1-t})dt\\ \nonumber
&=&0.
\end{eqnarray*}

For $t\in(1+\varepsilon,\infty)$, we use the same argument as for $t\in[0,1-\varepsilon)$ and obtain that
\begin{eqnarray*}
&&\frac{d}{d\lambda}\bigg|_{\lambda=1} \int_{1+\varepsilon}^\infty\log^-|\lambda t-1|\chi_{[0,1]}(\frac{t}{\lambda})+\log^-|\frac{t}{\lambda}-1|\chi_{[0,1]}(\lambda t)dt\\ \nonumber
&=&\int_{1+\varepsilon}^\infty\frac{d}{d\lambda}\bigg|_{\lambda=1}\log|1-\lambda t|+\log|1-\frac{t}{\lambda}|dt\\ \nonumber
&=&\int_{1+\varepsilon}^\infty\frac{d}{d\lambda}\bigg|_{\lambda=1}\log|1-\lambda t|dt\\ \nonumber
&=&\int_{1+\varepsilon}^\infty(\frac{-t}{1-t}+\frac{t}{1-t})dt\\ \nonumber
&=&0.
\end{eqnarray*}
And
\begin{eqnarray*}
&&\frac{d}{d\lambda}\bigg|_{\lambda=1} \int_0^\infty \chi_{[0,1]}(\lambda t)\chi_{[0,1]}(\frac{t}{\lambda})dt\\
&=&\frac{d}{d\lambda}\bigg|_{\lambda=1} \int_{1-\varepsilon}^{1+\varepsilon}\chi_{[0,1]}(\lambda t)\chi_{[0,1]}(\frac{t}{\lambda})dt\\
&=&\frac{d}{d\lambda}\bigg|_{\lambda=1} \int_{1-\varepsilon}^{\frac{1}{\lambda}}\chi_{[0,1]}(\lambda t)\chi_{[0,1]}(\frac{t}{\lambda})dt\\
&=&\frac{d}{d\lambda}\bigg|_{\lambda=1}\frac{1}{\lambda}\\
&=&-1.
\end{eqnarray*}
\end{proof}

\textbf{Proof of Proposition 2.2}
From Lemma \ref{lem2.4} , Lemma \ref{lem2.5} and Lemma \ref{lem2.6}, let $\varepsilon\rightarrow 0$, it's easy to see that
\begin{eqnarray*}
\frac{d}{d\lambda}\bigg|_{\lambda=1}\int_0^\infty\varphi(\lambda t)\psi(\frac{t}{\lambda})dt&=&-AC\pi^2-BD
-\int_0^{1}[th_0'(t)h_1(t)-th_0(t)h_1'(t)]dt\\
&&-\int^\infty_{1}[th_0'(t)h_1(t)-th_0(t)h_1'(t)]dt.
\end{eqnarray*}

\section{The proof of our main Theorem}
\textbf{Proof of Theorem 1.1}
The following argument is similar as in \cite{RS}, for reader's convenience, we prove it here.
For strictly star-shaped domains $\Omega\subset R^n$, we denote it's center by $z_0$. Let us first assume that $\Omega$ is strictly star-shaped with respect to the origin, that is, $z_0=0$.

We prove that
\begin{equation}\label{1.4}
\int_\Omega(x\cdot\nabla v)(-\triangle)^su=\frac{d}{d\lambda}\bigg|_{\lambda=1^+}\int_\Omega v_\lambda(-\triangle)^sudx,
\end{equation}
where $\frac{d}{d\lambda}\mid_{\lambda=1^+}$ is the derivative from the right side at $\lambda=1$. Indeed, let $g=(-\triangle)^su$. By Corollary 1.6 in \cite{RS2} and Proposition 1.6 in \cite{RS}, $g$ is defined pointwise in $\Omega$ and $g\in L^\infty(\Omega)$. Then, making the change of variables $y=\lambda x$ and using that supp $u_\lambda=\frac{1}{\lambda}\Omega\subset\Omega$, for $\lambda>1$, we obtain
\begin{eqnarray*}
&&\frac{d}{d\lambda}\bigg|_{\lambda=1^+}\int_\Omega v_\lambda g(x)dx\\
&=&\lim_{\lambda\downarrow1}\int_\Omega\frac{v(\lambda x)-v(x)}{\lambda-1}g(x)dx\\
&=&\lim_{\lambda\downarrow1}\lambda^{-n}\int_{\lambda\Omega}\frac{v(y)-v(y/\lambda)}{\lambda-1}g(y/\lambda)dy\\
&=&\lim_{\lambda\downarrow1}\int_\Omega\frac{v(y)-v(y/\lambda)}{\lambda-1}g(y/\lambda)dy+\lim_{\lambda\downarrow1}\int_{(\lambda\Omega)/\Omega}\frac{-v(y/\lambda)}{\lambda-1}g(y/\lambda)dy.
\end{eqnarray*}
By Lebesgue's dominated convergence theorem,
\begin{equation*}
\lim_{\lambda\downarrow1}\int_\Omega\frac{v(y)-v(y/\lambda)}{\lambda-1}g(y/\lambda)dy=\int_\Omega(y\cdot\nabla v)g(y)dy,
\end{equation*}
since $g\in L^\infty(\Omega)$, $\nabla v(\xi)\leq C\delta(\xi)^{s-1}\leq C\lambda^{1-s}\delta(y)^{s-1}$ for all $\xi$ in the line segment joining $y$ and $y/\lambda$, and $\delta^{s-1}$ is integrable. Then gradient bound $|\nabla v(\xi)|\leq C\delta(\xi)^{s-1}$ follows from assumption (a) in Corollary 1.6 of \cite{RS2} with $\beta=1$. Hence, to prove $(\ref{1.4})$ it remains only to show that
\begin{equation*}
\lim_{\lambda\downarrow1}\int_{(\lambda\Omega)/\Omega}\frac{-v(y/\lambda)}{\lambda-1}g(y/\lambda)dy=0
\end{equation*}
Indeed, $|(\lambda\Omega)\backslash\Omega|\leq C(\lambda-1)$ and by Corollary 1.6 in \cite{RS2} $v\in C^s(R^n)$ and $v\equiv0$ outside $\Omega$. Hence,
$$\|v\|_{L^\infty((\lambda\Omega)\backslash\Omega)}\rightarrow 0.$$

Now by the integration by parts formula,
\begin{eqnarray*}
\int_\Omega v_\lambda(-\triangle)^s udx&=&\int_{R^n} v_\lambda(-\triangle)^s udx\\
&=&\int_{R^n} (-\triangle)^{s/2}v_\lambda(-\triangle)^{s/2} udx\\
&=&\lambda^s\int_{R^n} (-\triangle)^{s/2}v(\lambda x)(-\triangle)^{s/2} u(x)dx\\
&=&\lambda^{\frac{2s-n}{2}}\int_{R^n}(-\triangle)^{s/2}v(\sqrt{\lambda} y)(-\triangle)^{s/2} u(\frac{1}{\sqrt{\lambda}}y)dy,
\end{eqnarray*}
here we use the change of variables $y=\sqrt{\lambda}x$.

Furthermore, this leads to
\begin{eqnarray*}\label{1.5}\nonumber
&&\int_\Omega(\nabla v\cdot x)(-\triangle)^sudx\\ \nonumber
&=&\frac{d}{d\lambda}\bigg|_{\lambda\rightarrow1^+}\{\lambda^{\frac{2s-n}{2}}\int_{R^n}(-\triangle)^{s/2}v(\sqrt{\lambda} y)(-\triangle)^{s/2} u(\frac{1}{\sqrt{\lambda}}y)dy\}\\  \nonumber
&=&\frac{2s-n}{2}\int_{R^n}(-\triangle)^{s/2}u(-\triangle)^{s/2}vdx+\frac{d}{d\lambda}\bigg|_{\lambda\rightarrow1^+}\int_{R^n}(-\triangle)^{s/2}v(\sqrt{\lambda} y)(-\triangle)^{s/2} u(\frac{1}{\sqrt{\lambda}}y)dy\\
&=&\frac{2s-n}{2}\int_{R^n}(-\triangle)^{s/2}u(-\triangle)^{s/2}vdx+\frac{1}{2}\frac{d}{d\lambda}\bigg|_{\lambda\rightarrow1^+}\int_{R^n}(-\triangle)^{s/2}v(\lambda y)(-\triangle)^{s/2} u(\frac{1}{\lambda}y)dy.
\end{eqnarray*}
Hence, what remains is to prove that
\begin{eqnarray}\label{1.6}\nonumber
\frac{d}{d\lambda}\bigg|_{\lambda\rightarrow1^+}I_\lambda&=&-\Gamma(1+s)^2\int_{\partial\Omega}\frac{u}{\delta^s}\frac{v}{\delta^s}(x\cdot\nu)d\sigma\\  \nonumber
&-&\int_{\Omega}\bigg(x\nabla U(x)V(x)-xU(x)\nabla V(x)\bigg)dx\\
&-&\int_{R^n\backslash\Omega}\bigg(x\nabla U(x)V(x)-xU(x)\nabla V(x)\bigg)dx,
\end{eqnarray}
where $U(x)=(-\triangle)^{s/2}u(x)$ and $V(x)=(-\triangle)^{s/2}v(x)$ and
\begin{eqnarray}\label{1.7}
I_\lambda=\int_{R^n}(-\triangle)^{s/2}v(\lambda y)(-\triangle)^{s/2} u(\frac{1}{\lambda}y)dy.
\end{eqnarray}

Now for each $\theta\in S^{n-1}$ there exists a unique $r_\theta>0$ such that $r_\theta\theta\in\partial\Omega$. Writing the integral $(\ref{1.7})$ in spherical coordinates and using the change of variables $t=r/{r_\theta}$, we have
\begin{eqnarray*}
&&\frac{d}{d\lambda}\bigg|_{\lambda\rightarrow1^+}I_\lambda\\
&=&\frac{d}{d\lambda}\bigg|_{\lambda\rightarrow1^+}\int_{S^{n-1}}d\theta\int_0^\infty r^{n-1}(-\triangle)^{s/2}v(\lambda r\theta)(-\triangle)^{s/2}u(\frac{r}{\lambda}\theta)dr\\
&=&\frac{d}{d\lambda}\bigg|_{\lambda\rightarrow1^+}\int_{S^{n-1}}r_\theta d\theta\int_0^\infty(r_\theta t)^{n-1}(-\triangle)^{s/2}v(\lambda r_\theta t\theta)(-\triangle)^{s/2}u(\frac{r_\theta t}{\lambda}\theta)dt\\
&=&\frac{d}{d\lambda}\bigg|_{\lambda\rightarrow1^+}\int_{\partial\Omega}(x\cdot\nu )d\sigma(x)\int_0^\infty t^{n-1}(-\triangle)^{s/2}v(\lambda tx)(-\triangle)^{s/2}u(\frac{tx}{\lambda})dt,
\end{eqnarray*}
where
\begin{equation*}
r_\theta^{n-1}d\theta=(\frac{x}{|x|}\cdot\nu)d\sigma=\frac{1}{r_\theta}(x\cdot\nu)d\sigma
\end{equation*}
Note that the change of variables $S^{n-1}\rightarrow\partial\Omega$ that maps every point in $S^{n-1}$ to its radial projection on $\partial\Omega$, and is unique because of the strict star-shapedness of $\Omega$.

Fix $x_0\in\partial\Omega$ and define
\begin{equation*}
\varphi(t)=t^{\frac{n-1}{2}}(-\triangle)^{s/2}u(tx_0),\,  \psi(t)=t^{\frac{n-1}{2}}(-\triangle)^{s/2}v(tx_0)
\end{equation*}
By Proposition \ref{pro1},
\begin{eqnarray*}
\varphi(t)=c_1\{\log^-\delta(tx_0)+c_2\chi_{[0,1]}\}\bar{u}(tx_0)+\bar{h}_0(t),\\
\psi(t)=c_1\{\log^-\delta(tx_0)+c_2\chi_{[0,1]})\}\bar{v}(tx_0)+\bar{h}_1(t),
\end{eqnarray*}
in $[0,\infty)$, where $\bar{u}$ is a $C^\alpha(R^n)$ extension of $u/{\delta^s}|_\Omega$ and $\bar{v}$ is a $C^\alpha(R^n)$ extension of $v/{\delta^s}|_\Omega$, $\bar{h}_0, \bar{h}_1$ are  $C^\alpha([0,\infty))$ functions. Next we will modify this expression in order to apply Proposition $\ref{pro2}$.

 Since $\Omega$ is $C^{1,1}$ and strictly star- shaped, it is not difficult to see that $\frac{|r-r_\theta|}{\delta(r\theta)}$ is a Lipschitz function of $r$ in $[0,\infty)$ and is bounded below by a positive constant (independent of $x_0$). Similarly, $\frac{|t-1|}{\delta(r\x_0)}$ and $\frac{\min\{|t-1|,1\}}{\min\{\delta(tx_0),1\}}$ are positive and Lipschitz functions of $t$ in $[0,\infty)$. Therefore,
\begin{equation*}
\log^-|t-1|-\log^-\delta(tx_0)
\end{equation*}
is Lipschitz in $[0,\infty)$ as a function of $t$.

Hence, for $t\in[0,\infty)$,
\begin{eqnarray*}
\varphi(t)=c_1\{\log^-|t-1|+c_2\chi_{[0,1]}\}\bar{u}(tx_0)+H_0(t)\\
\psi(t)=c_1\{\log^-|t-1|+c_2\chi_{[0,1]}\}\bar{v}(tx_0)+H_1(t)
\end{eqnarray*}
where $H_1$ $H_2$ are  $C^\alpha$ functions on the same interval.

Moreover, note that the difference
\begin{equation*}
\bar{u}(tx_0)-\bar u(x_0)
\end{equation*}
is $C^\alpha$ and vanishes at $t=1$. So is $\bar{v}(tx_0)-\bar v(x_0)$. Thus, for $t\in [0,\infty)$
\begin{eqnarray*}
\varphi(t)=c_1\{\log^-|t-1|+c_2\chi_{[0,1]}\}\bar{u}(x_0)+h_0(t),\\
\psi(t)=c_1\{\log^-|t-1|+c_2\chi_{[0,1]}\}\bar{v}(x_0)+h_1(t),
\end{eqnarray*}
where $h_0$ $h_1$ are $C^\alpha$ in $[0,\infty)$ . Therefore,
\begin{eqnarray*}
&&\frac{d}{d\lambda}\bigg|_{\lambda\rightarrow1^+}I_\lambda\\
&=&\frac{d}{d\lambda}\bigg|_{\lambda\rightarrow1^+}\int_{\partial\Omega}(x\cdot\nu )d\sigma(x)\int_0^\infty\varphi(\lambda t)\psi(\frac{t}{\lambda}) dt,
\end{eqnarray*}
and from Proposition \ref{pro2}, we know
\begin{eqnarray*}
&&\frac{d}{d\lambda}\bigg|_{\lambda\rightarrow1^+}\int_0^\infty\varphi(\lambda t)\psi(\frac{t}{\lambda}) dt\\
&=&-c_1^2(\pi^2+c_2^2)\frac{u}{\delta^s}\frac{v}{\delta^s}-\int_0^{1}[th_0'(t)h_1(t)-th_0(t)h_1'(t)]dt\\
&&-\int^\infty_{1}[th_0'(t)h_1(t)-th_0(t)h_1'(t)]dt,
\end{eqnarray*}
and
\begin{eqnarray*}
c_1=\frac{\Gamma(1+s)\sin(\frac{\pi s}{2})}{\pi} \,\,and\,\,c_2=\frac{\pi}{\tan(\frac{\pi s}{2})}.
\end{eqnarray*}
Therefore
\begin{eqnarray*}
c_1^2(\pi^2+c_2^2)&=&\frac{\Gamma(1+s)^2\sin^2(\frac{\pi s}{2})}{\pi^2}\bigg(\pi^2+\frac{\pi^2}{\tan^2(\frac{\pi s}{2})}\bigg)\\
&=&\Gamma(1+s)^2.
\end{eqnarray*}
Now we can express $\frac{d}{d\lambda}\bigg|_{\lambda\rightarrow1^+}I_\lambda$ with $u$ and $v$, that is,
\begin{eqnarray*}
\frac{d}{d\lambda}\bigg|_{\lambda\rightarrow1^+}I_\lambda
&=&\frac{d}{d\lambda}\bigg|_{\lambda\rightarrow1^+}\int_{\partial\Omega}(x\cdot\nu )d\sigma(x)\int_0^\infty\varphi(\lambda t)\psi(\frac{t}{\lambda}) dt\\
&=&-\Gamma(1+s)^2\int_{\partial \Omega}\frac{u}{\delta^s}\frac{v}{\delta^s}(x\cdot\nu)d\sigma\\
&-&\int_{\Omega}\bigg(x\nabla\big((-\triangle)^{s/2}u\big)(-\triangle)^{s/2}v-x(-\triangle)^{s/2}u\nabla\big((-\triangle)^{s/2}v\big)\bigg)dx\\
&-&\int_{R^n\backslash\Omega}\bigg(x\nabla\big((-\triangle)^{s/2}u\big)(-\triangle)^{s/2}v-x(-\triangle)^{s/2}u\nabla\big((-\triangle)^{s/2}v\big)\bigg)dx.
\end{eqnarray*}
We complete the proof.

\textbf{Proof of Theorem 1.2}
From Theorem $\ref{thm2}$, we know
\begin{eqnarray*}
&&\int_{\Omega}(x\cdot\nabla v)(-\triangle)^s u\\
&=&\frac{2s-n}{2}\int_{\Omega}v(-\triangle)^s udx-\frac{\Gamma(1+s)^2}{2}\int_{\partial \Omega}\frac{u}{\delta^s}\frac{v}{\delta^s}(x\cdot\nu)d\sigma\\
&&-\frac{1}{2}\int_{\Omega}\bigg(x\nabla\big((-\triangle)^{s/2}u\big)(-\triangle)^{s/2}v-x(-\triangle)^{s/2}u\nabla\big((-\triangle)^{s/2}v\big)\bigg)dx\\
&&-\frac{1}{2}\int_{R^n\backslash\Omega}\bigg(x\nabla\big((-\triangle)^{s/2}u\big)(-\triangle)^{s/2}v-x(-\triangle)^{s/2}u\nabla\big((-\triangle)^{s/2}v\big)\bigg)dx.
\end{eqnarray*}
And using the integration by parts formula, we obtain
$$\int_{\Omega}|x|^av^p(x\cdot\nabla\nu)dx=-\int_{\Omega}\frac{n+a}{p+1}|x|^av^{p+1}dx.$$
Therefore,
\begin{eqnarray}\label{2.1}\nonumber
&&(s-\frac{n}{2})\int_{\Omega}v(-\triangle)^sudx-\frac{\Gamma(1+s)^2}{2}\int_{\partial \Omega}\frac{u}{\delta^s}\frac{v}{\delta^s}(x\cdot\nu)d\sigma\\ \nonumber
&-&\frac{1}{2}\int_{\Omega}\bigg(x\nabla\big((-\triangle)^{s/2}u\big)(-\triangle)^{s/2}v-x(-\triangle)^{s/2}u\nabla\big((-\triangle)^{s/2}v\big)\bigg)dx\\ \nonumber
&-&\frac{1}{2}\int_{R^n\backslash\Omega}\bigg(x\nabla\big((-\triangle)^{s/2}u\big)(-\triangle)^{s/2}v-x(-\triangle)^{s/2}u\nabla\big((-\triangle)^{s/2}v\big)\bigg)dx\\
&=&-\int_{\Omega}\frac{n+a}{p+1}|x|^av^{p+1}dx.
\end{eqnarray}
Let $u=v$ and $v=u$, for the second equation in problem (\ref{1.1}), we have
\begin{eqnarray}\label{2.2}\nonumber
&&(s-\frac{n}{2})\int_{\Omega}u(-\triangle)^svdx-\frac{\Gamma(1+s)^2}{2}\int_{\partial \Omega}\frac{v}{\delta^s}\frac{u}{\delta^s}(x\cdot\nu)d\sigma\\ \nonumber
&-&\frac{1}{2}\int_{\Omega}\bigg(x\nabla\big((-\triangle)^{s/2}v(x)\big)(-\triangle)^{s/2}u(x)-x(-\triangle)^{s/2}v(x)\nabla\big((-\triangle)^{s/2}u(x)\big)\bigg)dx\\ \nonumber
&-&\frac{1}{2}\int_{R^n\backslash\Omega}\bigg(x\nabla\big((-\triangle)^{s/2}v(x)\big)(-\triangle)^{s/2}u(x)-x(-\triangle)^{s/2}v(x)\nabla\big((-\triangle)^{s/2}u(x)\big)\bigg)dx\\
&=&-\int_{\Omega}\frac{n+b}{q+1}|x|^bu^{q+1}dx.
\end{eqnarray}
Adding up (\ref{2.1}) and (\ref{2.2}), we obtain the Pohozaev identity of the problem (\ref{1.1}),
\begin{eqnarray}\label{2.3}\nonumber
&&(s-\frac{n}{2})\int_{\Omega}u(-\triangle)^svdx+(s-\frac{n}{2})\int_{\Omega}v(-\triangle)^sudx-\Gamma(1+s)^2\int_{\partial \Omega}\frac{v}{\delta^s}\frac{u}{\delta^s}(x\cdot\nu)d\sigma\\
&=&-\int_{\Omega}\frac{n+b}{q+1}|x|^bu^{q+1}dx-\int_{\Omega}\frac{n+a}{p+1}|x|^av^{p+1}dx.
\end{eqnarray}

Since that $\Omega$ is a star-shaped domain, we must have $x\cdot\nu>0$. For $u, v>0$, it holds that
\begin{eqnarray*}
-\Gamma(1+s)^2\int_{\partial \Omega}\frac{v}{\delta^s}\frac{u}{\delta^s}(x\cdot\nu)d\sigma< 0.
\end{eqnarray*}
Hence by $(\ref{1.1})$ and $(\ref{2.3})$
\begin{eqnarray}\label{2.4}\nonumber
&&(s-\frac{n}{2})\int_{\Omega}u(-\triangle)^svdx+(s-\frac{n}{2})\int_{\Omega}v(-\triangle)^sudx\\ \nonumber
&=&(s-\frac{n}{2})\int_{\Omega}|x|^bu^{q+1}dx+(s-\frac{n}{2})\int_{\Omega}|x|^av^{p+1}dx\\
&>&-\frac{n+b}{q+1}\int_{\Omega}|x|^bu^{q+1}dx-\frac{n+a}{p+1}\int_{\Omega}|x|^av^{p+1}dx.
\end{eqnarray}
Since $u$ and $v\in H^s(R^n)$, we know that
\begin{eqnarray*}
\int_{\Omega}|x|^bu^{q+1}dx=\int_{R^n}|x|^bu^{q+1}dx=\int_{R^n}u(-\triangle)^svdx=\int_{R^n}(-\triangle)^{\frac{s}{2}}u(-\triangle)^{\frac{s}{2}}vdx,
\end{eqnarray*}
and
\begin{eqnarray*}
\int_{\Omega}|x|^av^{p+1}dx=\int_{R^n}|x|^av^{p+1}dx=\int_{R^n}v(-\triangle)^sudx=\int_{R^n}(-\triangle)^{\frac{s}{2}}u(-\triangle)^{\frac{s}{2}}vdx,
\end{eqnarray*}
thus
\begin{equation}\label{2.5}
\int_{\Omega}|x|^av^{p+1}dx=\int_{\Omega}|x|^bu^{q+1}dx.
\end{equation}
If $\frac{n+a}{p+1}+\frac{n+b}{q+1}\leq n-2s$, then
\begin{eqnarray*}
&&(s-\frac{n}{2}+\frac{n+b}{q+1})\int_\Omega|x|^bu^{q+1}dx+(s-\frac{n}{2}+\frac{n+a}{p+1})\int_\Omega|x|^av^{p+1}dx\\
&=&(2s-n+\frac{n+b}{q+1}+\frac{n+a}{p+1})\int_\Omega|x|^bu^{q+1}dx\\
&\leq&0,
\end{eqnarray*}
this is a contradiction with $(\ref{2.4})$. This completes our proof.

\noindent{\bf Acknowledgement}

The research was supported by NSFC(NO.11571176) and Natural Science Foundation
of the Jiangsu Higher Education Institutions (No.14KJB110017).
The authors would like to express sincere thanks to the anonymous
referee for his/her carefully reading the manuscript and valuable
comments and suggestions.

\bigskip

{\em Author's Addresses and Emails:}
\medskip

Pei Ma

Jiangsu Key Laboratory for NSLSCS

School of Mathematical Sciences

Nanjing Normal University

Nanjing, Jiangsu 210023, China;

Department of Mathematical Sciences

Yeshiva University

New York, NY, 10033, USA

mapei0620@126.com

\medskip

Fengquan Li

School of Mathematical Sciences

Dalian University of Technology

Dalian, Liaoning 116024, China

fqli@dlut.edu.cn

\medskip

Yan Li

Department of Mathematical Sciences

Yeshiva University

New York, NY, 10033, USA

yali3@mail.yu.edu

\end{document}